\documentclass[11pt]{amsart}
\usepackage{amsmath,graphicx,latexsym}
\usepackage{amsmath,amssymb,bm, color}

\begin{document}
\newtheorem{theorem}{Theorem}[section]
\newtheorem{lemma}[theorem]{Lemma}
\newtheorem{corollary}[theorem]{Corollary}
\newtheorem{prop}[theorem]{Proposition}
\newtheorem{definition}[theorem]{Definition}
\newtheorem{remark}[theorem]{Remark}


 \def\ad#1{\begin{aligned}#1\end{aligned}}  \def\b#1{{\bf #1}} \def\hb#1{\hat{\bf #1}}
\def\a#1{\begin{align*}#1\end{align*}} \def\an#1{\begin{align}#1\end{align}}
\def\e#1{\begin{equation}#1\end{equation}} \def\t#1{\hbox{\rm{#1}}}
\def\dt#1{\left|\begin{matrix}#1\end{matrix}\right|}
\def\p#1{\begin{pmatrix}#1\end{pmatrix}} \def\op#1{\operatorname{#1}}
\def\bS#1{\overline{\mathcal{#1}}_h}

 \numberwithin{equation}{section} \def\P{\Pi_h^\nabla}

\title  [P1 honeycomb virtual element ]
   {Superconvergent P1 honeycomb virtual elements and lifted P3 solutions}

\author {Yanping Lin}
\address{Department of Applied Mathematics, The Hong Kong Polytechnic University, Hung Hom, Hong Kong}
\email{yanping.lin@polyu.edu.hk}
\thanks{Yanping Lin is supported in part by HKSAR GRF 15302922  and polyu-CAS joint Lab.}

\author {Xuejun Xu}
\address{School of Mathematical Science, \ Tongji University, \ Shanghai, \ 200092, \ China}
	\address{Institute of Computational Mathematics, AMSS, Chinese Academy of Sciences, Beijing, 100190, China}
\email{xxj@lsec.cc.ac.cn}
\thanks{Xuejun Xu is supported by National Natural Science Foundation of China (Grant
Nos. 12071350), Shanghai Municipal Science and Technology Major Project No.
2021SHZDZX0100, and Science and Technology Commission of Shanghai Municipality.}

\author { Shangyou Zhang }
\address{Department of Mathematical
            Sciences, University
     of Delaware, Newark, DE 19716, USA. }
\email{szhang@udel.edu }

\date{}

\begin{abstract} 
When solving the Poisson equation on
   honeycomb hexagonal grids, we show that the $P_1$ virtual element is three-order 
  superconvergent in $H^1$-norm, and two-order superconvergent in $L^2$ and   $L^\infty$
  norms.
We define a local post-process which lifts the superconvergent $P_1$ solution to a 
   $P_3$ solution of the optimal-order approximation.
The theory is confirmed by a numerical test.

\end{abstract}



\maketitle \baselineskip=14pt

\section{Introduction}
In this work, we study  the
  $P_1$ virtual element on honeycomb meshes for solving the following Poisson equation. 
\an{ \label{p-e} \ad{ -\Delta  u & = f  && \t{in } \ \Omega, \\
            u&=0 && \t{on } \ \partial \Omega, } }
where  $f\in L^2(\Omega)$ and $\Omega\subset \mathbb{R}^2$ is a bounded polygonal domain 
which can be meshed by one-size honeycomb hexagons and a few additional polygons consisting
   of equilateral triangles only at domain boundary, such as the
   one illustrated in Figure \ref{f-grid}.
\begin{figure}[h!] 
\begin{center}\begin{picture}(410,100)(0,0)
  \put(-5,-170){\includegraphics[width=410pt]{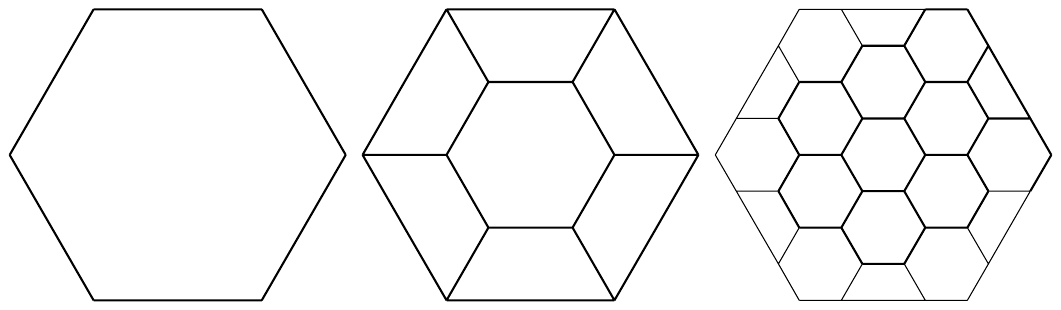}}
\end{picture} 
\end{center}
\caption{ The first three honeycomb grids on a regular hexagon domain $\Omega$. }   
\label{f-grid} 
\end{figure}

The virtual element method is proposed and studied in  
\cite{Beirao,Beirao16, Cao-Chen, Cao-Chen2,  Chen1, Chen2, Chen3,
  Feng-Huang,Feng-Huang2, Huang1, Huang2, Huang3, Wang-Mu, Wang-Wang}. 
In 2D, the virtual element functions are continuous $P_k$ polynomials on 
   edges of a polygonal mesh and extended to inside polygons by the Poisson equation with
   $P_{k-2}$ moments.
In particular, for the 2D $P_1$ virtual elements studied in this paper,
   the continuous $P_1$ polynomials on the edges of the mesh are harmonically
  extended into polygons.
Such non-polynomial virtual element functions are interpolated into
  a  computable polynomial space on each polygon.
As the interpolation is usually not a full-rank operator,
  a stabilizer has to be added to the discrete equation.
However, simply raising the polynomial degree of the local interpolation space,
   in order to  eliminate the stabilizer, 
  does not work in general, cf. \cite{Xu-Z}.

The references \cite{Berrone-0,Berrone-1} propose to use some high-degree polynomial 
   interpolation space only for 2D $P_1$ virtual elements to get some 
 stabilizer-free virtual element methods.
The reference \cite{Berrone} proposes to use only high-$P_k$ harmonic polynomials as
  the interpolation functions in a stabilizer-free virtual element method.
The references \cite{Lin,Xu-Z} use multi-piece polynomials as the interpolation space on 
 each polygon or polyhedron,  so that the resulting virtual element method
  works for any $P_k$ virtual element without using stabilizer.

In this work,  we apply the stabilizer-free $P_1$ virtual element method \cite{Lin,Xu-Z}
  to the Poisson equation \eqref{p-e} on 2D honeycomb meshes, 
   cf. Figure \ref{f-grid}.
Taking the advance of extreme symmetry, we show that the $P_1$ virtual element solution
 $u_h$ is
  three-order superconvergent (three orders above the optimal order) in $H^1$-norm, i.e.,
\an{\label{inf} \|\nabla (I_h u-u_h)\|_{L^2}\le C h^4|u|_{H^{4,\infty}}, }
where $H^{4,\infty}$ is the Sobolev space of functions with all 4-th
   derivatives bounded, and $I_h$ is the $P_1$ nodal interpolation at corners
   of honeycombs. 
Additionally,   the $P_1$ virtual element 
  solution is two-order superconvergent in $L^2$-norm and in $L^\infty$ norm, i.e., 
\an{\label{inf2} \|I_h u-u_h\|_{L^2}\le
  C\| I_h u-u_h\|_{L^\infty}\le C h^4|u|_{H^{4,\infty}}. }
We then define a local-lift post-process,  which lifts the $P_1$ virtual element
   solution $u_h$ to a piecewise $P_3$ solution 
   $\tilde u_h$ of the optimal order approximation, i.e.,  
\an{\label{3o} \| u-\tilde u_h\|_{L^2} + h |u-\tilde u_h|_{H^1,h}
    \le C h^4 |u|_{H^{4,\infty}}, }
where $|\cdot|_{H^1,h}$ is a piecewise $H^1$-semi-norm based on a 
  patch grid $\mathcal{R}_h$  (defined in section 3.)
We remark that the $P_1$ virtual element solution itself can only converge at the optimal order, 
two orders lower than that of \eqref{3o},
\a{ \| u- u_h\|_{L^2} + h |u- u_h|_{H^1,h}
    \le C h^2 |u|_{H^{2}}. }
The theory is confirmed by a numerical test.
  
Some early work and influential references on the superconvergence of $P_1$ finite elements 
   can be found in
 \cite{Blum-Lin,Bramble,Chen-Wang,Douglas,Huang-Zhang,Krizek,Schatz,Wahlbin,Wang,Wang-J,
  Zlamal,Zhu-Lin,Zienkiewicz}.
Some superconvergence results for weak Galerkin and conforming discontinuous 
 Galerkin finite element methods can be found in \cite{Al-Taweel-Z21,Wang-Z23,
Wang-Z23a,Ye-Z21b,Ye-Z21c,Ye-Z21e,Ye-Z21g,Ye-Z22c,Ye-Z22d,Ye-Z22e,
Ye-Z23b,Ye-Z23c,Ye-Z23d,Ye-Z23e}. 
                \vskip .7cm
 
\section{The $P_1$ honeycomb virtual element} 

We define the $P_1$ honeycomb virtual element in this section and show
  that the $P_1$ honeycomb virtual  element equation has a unique solution.

\def\TT{\overline{\mathcal{T}}_h}

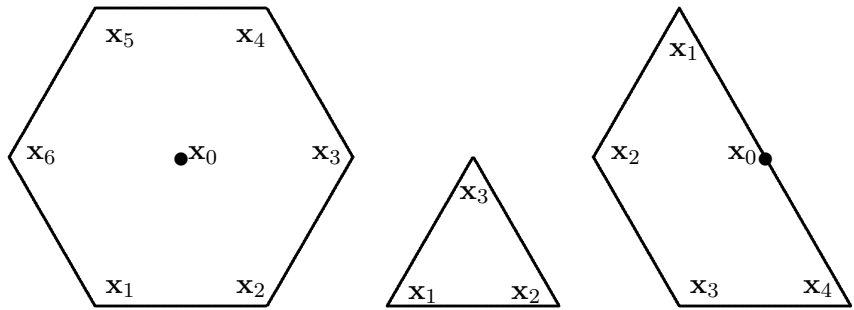
\begin{figure}[ht] \setlength{\unitlength}{1.3pt}

 \begin{center}\begin{picture}( 250.,100)( -90.,  0.)
     \def\lb{\circle*{0.8}}\def\lc{\vrule width1.2pt height1.2pt}
     \def\la{\circle*{0.3}}
\put(-90,0){\begin{picture}(  100.,  86.6025391)( -50.,  43.3012695)
     \def\lb{\circle*{0.8}}\def\lc{\vrule width1.2pt height1.2pt}
     \def\la{\circle*{0.3}} \put(0,86){\circle*{4}} \put(2,86){$\b x_0$}
     \put(-45,86){$\b x_6$} \put(-22,47){$\b x_1$} \put( 16,47){$\b x_2$}
     \put(38,86){$\b x_3$} \put(-22,120){$\b x_5$} \put( 16,120){$\b x_4$}
    \multiput( -25.00, 129.90)(   0.250,   0.000){200}{\la}
     \multiput( -50.00,  86.60)(   0.125,   0.217){200}{\la}
     \multiput(  50.00,  86.60)(  -0.125,   0.217){200}{\la}
     \multiput(  50.00,  86.60)(  -0.125,  -0.217){200}{\la}
     \multiput( -50.00,  86.60)(   0.125,  -0.217){200}{\la}
     \multiput( -25.00,  43.30)(   0.250,   0.000){200}{\la}
 \end{picture}}
\put(20,0){\begin{picture}(  50.,  43.3012695)( -25.,  129.903809)
     \def\lb{\circle*{0.8}}\def\lc{\vrule width1.2pt height1.2pt}
     \def\la{\circle*{0.3}}
         \put(-4,161){$\b x_3$}
        \put(-19,132){$\b x_1$} \put( 11,132){$\b x_2$}

     \multiput( -25.00, 129.90)(   0.125,   0.217){200}{\la}
     \multiput(  25.00, 129.90)(  -0.125,   0.217){200}{\la}
     \multiput( -25.00, 129.90)(   0.250,   0.000){200}{\la}

 \end{picture}}
\put(80,0){\begin{picture}(  75.,  86.6025391)( -50.,  43.3012695)
     \def\lb{\circle*{0.8}}\def\lc{\vrule width1.2pt height1.2pt}
     \def\la{\circle*{0.3}} \put(0,86){\circle*{4}} \put(-11,86){$\b x_0$}
     \put(-45,86){$\b x_2$} \put(-22,47){$\b x_3$} \put( 11,47){$\b x_4$}
      \put(-28,116){$\b x_1$}   
     \multiput( -25.00, 129.90)(   0.125,  -0.217){400}{\la}
     \multiput( -50.00,  86.60)(   0.125,   0.217){200}{\la}
     \multiput( -50.00,  86.60)(   0.125,  -0.217){200}{\la}
     \multiput( -25.00,  43.30)(   0.250,   0.000){200}{\la}

 \end{picture}}
 \end{picture}\end{center}\caption{Left: A regular hexagon inside $\Omega$, where
  $\b x_0$ is an auxiliary vertex of triangular mesh $\TT$;
   Middle: A regular triangle which can be only at a corner of $\Omega$;
   Right:  A pentagon consisting of three regular triangles which can be at an edge of $\Omega$
    with $\b x_1\b x_4\subset \partial\Omega$. }   
\label{f-3} 
\end{figure}  

Let $\mathcal{T}_h=\{ T \}$ be a uniform honeycomb mesh on  the domain $\Omega$,
 where $T$ is either a regular hexagon of edge-size $h$ inside $\Omega$,
  or a regular triangle of edge-side $h$ at a corner of the polygonal domain, or
  a pentagon consisting of three regular triangles of edge-side $h$ at an edge of domain, 
   cf. Figure \ref{f-3}.
We note that there is a vertex added at the middle of straight edge $\b x_1\b x_4$,
   cf. Figure \ref{f-3}.
For each hexagon $T$, we add the central point $\b x_0$ to it so that it is subdivided 
   into 6 regular triangles $\{ T_i\}$, cf. the left diagram of Figure \ref{f-3}.
For a pentagon $T\in \mathcal{T}_h$, we connect its center point $\b x_0$ with
  two vertices $\b x_2$ and $\b x_3$ so that
  it is subdivided into 3 regular triangles, cf. the right diagram in Figure \ref{f-3}.
In addition to the regular triangles in $ \mathcal{T}_h$ (the
  middle graph in Figure \ref{f-3}),  all theses regular triangles form a
  equilateral triangular mesh $\TT$ on domain $\Omega$.

 Let $\mathcal{E}_h$ denote the set of edges $e$ in $\mathcal{T}_h$.
 Let $E=\cup_{e\in \mathcal{E}_h} e$ and $\mathbb{B}_1(E)$ be the space of 
  piecewise linear functions on domain $E$,
\a{ \mathbb{B}_1(E) =\{ \tilde v \in C^0(E) : \tilde v|_e\in P_1(e) \ \t{ for all } e\in
   \mathcal{E}_h\}.  }
The virtual element space on the mesh $\mathcal{T}_h$ is defined as
\an{ \label{t-V-h} \ad{ \tilde 
    V_h=\{\tilde  v\in H^1_0(\Omega) & : \tilde v \in \mathbb{B}_k(\mathcal{E}_h ), \
   \Delta \tilde v|_K=0 \quad \forall T\in\mathcal{T}_h   \} .  } }

In computation, the non-polynomial virtual element function is interpolated to
  a piecewise polynomial.  
The piecewise polynomial space on each $T\in \mathcal{T}$ is defined by, cf.
  Figure \ref{f-3},
\a{  \mathbb{V}_1(T)=\begin{cases} \{ v_h \in H^1(T) : v_h|_{T_i}\in P_1(T_i) \ \t{if $
            T=\bigcup_{i=1}^6 T_i $ is a hexagon}\}, \\ 
        \{ v_h \in H^1(T) : v_h|_{T_1}\in P_1(T_1) \ \t{if 
           $ T= T_1$ is a triangle} \}, \\
       \{ v_h \in H^1(T) : v_h|_{T_i}\in P_1(T_i) \ \t{if $
            T= \bigcup_{i=1}^4 T_i$ is a pentagon} \}, \end{cases} }
where $T_i$ is an edge-size $h$ regular triangle.
        
The interpolated virtual element space on $\mathcal{T}_h$ is defined by
\an{ \label{V-h} V_h = \{ v_h=\Pi_h^\nabla \tilde v \ : \ v_h|_T \in \mathbb{V}_1(T), 
  \ T\in\mathcal{T}_h; \
   \tilde v\in \tilde V_h \}, }
where $v_h=\Pi_h^\nabla \tilde v$ is the local $H^1$-projection, i.e., $v_h$ is
   the local solution on $T\in\mathcal{T}_h$:   
\a{ \left\{ \ad{ (\nabla(v_h-\tilde v), \nabla w_h)_T&=0\quad \forall 
     w_h=\Pi_h^\nabla \tilde w \in \mathbb{V}_1(T), \\
           \langle v_h-\tilde v, w_h\rangle_{\partial T}& =0 \quad
     \forall w_h\in \mathbb{V}_1(T). } \right. }
The stabilizer-free virtual element equation reads:  
   Find $u_h=\P \tilde u\in V_h$ such that
\an{\label{f-e} (\nabla u_h,\nabla v_h)  = (f,v_h) 
     \quad \forall \tilde v\in \tilde V_h, \ v_h=\P \tilde v. }

\begin{lemma} The $P_1$ honeycomb virtual element equation \eqref{f-e} has a unique
  solution.
\end{lemma}

\begin{proof} \eqref{f-e} is a finite dimensional square system of linear equations.
Its existence of solution is implied by the uniqueness of the solution.
Let $f=0$ and $v_h=u_h$ in \eqref{f-e},  where $u_h$ is a solution.
\eqref{f-e} becomes $\|\nabla u_h\|_T^2=0$ on every $T\in \mathcal{T}_h$.
Thus $u_h\in P_0(T)$.  Because $u_h\in C^0(\Omega)$, $u_h=C$ is a global constant on
  $\Omega$. By $u_h|_{\partial \Omega}=0$, $u_h=0$.  The proof is complete.
\end{proof}

\section{The superconvergence in $H^1$-norm} 

We show in this section that the $P_1$ honeycomb virtual  element is
 three-order superconvergent in $H^1$-norm first, and is two-order superconvergent in 
$L^2$-norm and in $L^\infty$-norm.

\begin{theorem} Let $u\in H^{4,\infty}(\Omega)\cap H^1_0(\Omega)$ and $u_h\in V_h$ of \eqref{V-h} 
  be the solutions of Poisson's equation \eqref{p-e} and
  the finite element equation \eqref{f-e}, respectively.
  It holds that
\a{      | I_h u-u_h |_{H^1}\le C h^4|u|_{H^{4,\infty}}. } 
\end{theorem}

\begin{proof} For analysis,  we introduce the auxiliary, conforming $P_1$ finite element
   on a sub-triangular mesh of $\mathcal{T}_h$.
    Let $\mathcal{T}_h^6$ be the set of all hexagons in $\mathcal{T}_h$.
For each $T\in \mathcal{T}_h^6$, we introduce a new vertex/node $\b x_0$, the center 
  of the regular hexagon.
  Let $\mathcal{N}_h$ be the set of all vertices/nodes in $\mathcal{T}_h$.
Let 
\an{\label{b-N} \bS N= \mathcal{N}_h \bigcup \{ \b x_0 : \b x_0\in T \ \forall T\in 
    \mathcal{T}_h^6\}. }
Subdividing all hexagons and pentagons in $\mathcal{T}_h$ by
   connecting $\b x_0$ with all $\b x_i$, cf. Figure \ref{f-3},
  we have an equilateral triangular mesh on $\Omega$,
\an{\label{b-T} \bS T=\{ T : T \ \t{ is a regular triangle of edge-size $h$}\}.
   }
The auxiliary, conforming $P_1$ finite element
   space on triangular mesh $\bS T$ is defined by
\an{\label{b-V} \overline{V}_h =\{ v_h\in H^1_0(\Omega) : v_h|_T \in P_1(T) \
           \forall T\in \bS T\}.  }

Let $e_h=I_h u-u_h$, where $I_h : H^{2}(\Omega)\cap H^1_0(\Omega)
    \to V_h$ is the nodal interpolation
  operator based on nodal set $\mathcal{N}_h$.
Let $\overline{I}_h  : H^{2}(\Omega)\cap H^1_0(\Omega)
    \to \overline{V}_h$ (defined in \eqref{b-V}) be the nodal interpolation
  operator based on nodal set $\bS N$  (defined in \eqref{b-N}), i.e.,
\a{ I_u u(\b x_j) = u(\b x_j) \quad\forall \b x_j \in \bS N. }
Let 
\an{\label{d-h} d_h = I_h u - \overline{I}_h u \in \overline{V}_h.  }
By the harmonic solution in \eqref{V-h}, $d_h$ is supported only inside all
   hexagons $T\in \mathcal{T}^6_h$ and, cf Figure \ref{f-3}, due to a 
  regular hexagon,
\a{ d_h(\b x_0) &= \frac 16 \sum_{i=1}^6 u(\b x_i) - u(\b x_0),  \\
    d_h(\b x_i) &= 0, \quad i=1,\dots, 6.  }

Because
\an{\label{o} (\nabla u-\nabla u_h, \nabla v_h)=0 \quad\forall v_h\in V_h, }
it follows that by the divergence theorem,
\an{\label{m} \ad{ | I_h u-u_h |_{H^1}^2 &= ( \nabla(I_h u-u), \nabla e_h)\\ 
&= ( \nabla d_h, \nabla e_h) + ( \nabla(\overline{I}_h u-u), \nabla e_h)\\ 
   &= \sum_{T\in\mathcal{T}_h^6} ( \nabla d_h, \nabla e_h)_T + 
   \sum_{T\in\bS T} ( \nabla(\overline{I}_h u-u), \nabla e_h)_T \\
    & = 0 + \sum_{T\in\bS T}\sum_{i=0}^2
       \int_{\b x_i\b x_{i+1}}  (\overline{I}_h u-u)
             \partial_{\b t_{i+2}^\perp} e_h d\b s, }
} where $\b t_i^\perp$ is the $90$ degrees clockwise rotation of $\b t_i$,
   and $\b x_i$ are three vertices of $T$ depicted in Figure \ref{T}.
Here $( \nabla d_h, \nabla e_h)_T=0$ because $e_h$ is discrete harmonic inside
  a hexagon $T$ by \eqref{V-h}, and $d_h$ is supported inside $T$.

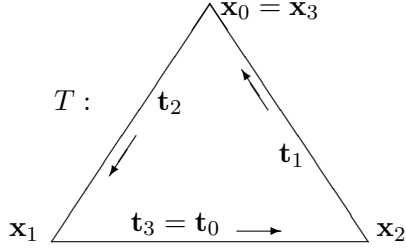
\begin{figure}[ht] \centering 
\begin{picture}(120,90)(0,0) \put(-16,2){$\b x_1$} \put(123,2){$\b x_2$} 
\put(64,86){$\b x_0=\b x_3$} 
 \put(1,50){$T:$} \put(40,50){$\b t_2$} \put(70,4){\vector(1,0){17}}  
   \put(32,40){\vector(-2,-3){10}} \put(82,50){\vector(-2,3){10}}
   \put(86,30){$\b t_1$}  \put(30,4){$\b t_3=\b t_0$}
  \put(0,0){\line(2,3){60}}\put(120,0){\line(-2,3){60}}\put(0,0){\line(1,0){120}}
\end{picture} \caption{A equilateral triangle with three unit tangent vectors.}
  \label{T}
\end{figure}

For linear polynomials, the normal derivative is a linear combination of two
  tangential derivatives.
We have, due to equilateral triangles, cf. Figure \ref{T},
\an{\label{m-0} \ad{ &\quad \  \sum_{T\in\bS T}\sum_{i=0}^2
       \int_{\b x_i\b x_{i+1}}  (\overline{I}_h  u-u) \partial_{\b t_{i-1}^\perp} e_h d\b s\\
  &=  \sum_{T\in\bS T}\sum_{i=0}^2
       \int_{\b x_i\b x_{i+1}}  (\overline{I}_h  u-u)
     \frac {2}{\sqrt{3}}
   (  \partial_{\b t_{i+1}} - \frac 1 2 \partial_{\b t_{i-1}}) e_h d\b s\\
  &=  \sum_{T\in\bS T}\sum_{i=0}^2
       \int_{\b x_i\b x_{i+1}}  (\overline{I}_h  u-u)
       \frac 2{\sqrt{3}} \partial_{\b t_{i+1}}  e_h d\b s, } }
   where the second term cancels itself from the other side of triangle
  because $\b t_{i+1}$ has an opposite direction on the other side
  of edge $\b x_i\b x_{i+1}$, or $e_h=0$ if $\b x_i\b x_{i+1}$ has no
  other neighbor triangle, i.e., a boundary edge. 

A version of the Euler-Maclaurin formula is, cf. \cite{Kac},
\a{ \int_a^b f(t) dt &= \sum_{n=a}^b f(n) -\frac{f(a)+f(b)}2
    -\sum_{i=1}^k \frac{b_{i}}{i!}
    (f^{(i-1)}(b)-f^{(i-1)}(a)) \\
    & \quad \ - \int_{a}^b P_{k+1} (t)f^{(k+1)}(t) dt, }
where $b_{i}$ is a Bernoulli number and $P_{k+1}(t)$ is
  a scaled  periodic Bernoulli polynomial $B_{k+1}(\{1-t\})/((k+1)!)$.
Other than $b_1$, all odd Bernoulli numbers are zero.
Letting $k=3$, $a=0$ and $b=1$,
  we 
have \a{ \int_0^1 f dt &= h \frac{f(0)+f(1)}2
    + \frac{ h^{2} }{4} \int_0^1 f'' dt 
       -  h^{4}\int_{0}^1 P_{4} (t)f^{(4)}(t) dt.} 
Noting that $\int_{\b x_1\b x_2} \overline{I}_h  u d\b s=(u(\b x_1)+u(\b x_2))h/2$,
we get from \eqref{m} and \eqref{m-0},  
\an{\label{m1} \ad{ |e_h|_{H^1}^2 &= -\sum_{T\in\bS T}\sum_{i=0}^2
       \frac { h^2} {2 \sqrt{3}} \int_{\b x_i\b x_{i+1}} \partial_{\b t_{i-1}}^2 u
       \partial_{\b t_{i+1}}  e_h d\b s\\
  &\quad \ + \sum_{T\in\bS T}\sum_{i=0}^2
       h^4 \int_{\b x_i\b x_{i+1}}   P_4(\b x) \partial^4_{\b t_0} u
         \partial_{\b t_{i+1}}  e_h d\b s, \\
  &=:  \frac { h^2} {2\sqrt{3}} I_1+I_2. } }
By Cauchy-Schwarz inequality,  we get, because $P_4$ has a zero-point inside the edge
  ($\int_{\b x_i\b x_{i+1}}   P_4(\b x) d\b s=0$),
\an{\label{m-i-2}\ad{
 |I_2| & \le C h^4 |u|_{H^{4,\infty} }
            \bigg( \sum_{T\in \bS T}  |e_h|_{H^{1}(T)}^2\bigg)^{1/2}   
  = C h^4  |u|_{H^{4,\infty} }|e_h|_{H^{1}}. } }

For the first term in \eqref{m1}, we consider an area-integral, cf. Figure \ref{T},
 by integration by parts, noting that $\partial_{\b t_2}e_h$ is a constant,
\a{&\quad \
  \int_{T} (\partial_{\b t_1}\partial_{\b t_0}^2 u) \partial_{\b t_2}e_h d\b x  
   =\int_{\b x_1\b x_0} \partial_{\b t_0}^2 u \partial_{\b t_2} e_h d\b s
   -\int_{\b x_1\b x_2} \partial_{\b t_0}^2 u \partial_{\b t_2} e_h d\b s. }
When summing the first term above,  as $\b t_2$ is tangential to $\b x_1\b x_0$,
  the term is canceled with the term on the other side of the triangle.
Thus, we have, from \eqref{m1},
\a{ I_1 &= \sum_{T\in\bS T}\sum_{i=0}^2 \int_T(\partial_{\b t_{i}}
          \partial_{\b t_{i-1}}^2 u) \partial_{\b t_{i+1}}  e_h d\b x. }
Taking integration by parts, we get, noting that the angle between two $\b t_i$ is
  $120^0$, 
\an{\label{m2} \ad{ I_1 &= \sum_{T\in\bS T}\sum_{i=0}^2
   \bigg( -\int_T(\partial_{\b t_{i}}
          \partial_{\b t_{i+1}} \partial_{\b t_{i-1}}^2 u)   e_h d\b x \\
     &\ \quad   +\frac 12 \int_{\b x_i\b x_{i+1}}(\partial_{\b t_{i}}
          \partial_{\b t_{i-1}}^2 u) e_h  d\b s  
      +\frac 12 \int_{\b x_{i+1}\b x_{i-1}} (\partial_{\b t_{i}}
          \partial_{\b t_{i-1}}^2 u) e_h d\b s \bigg)\\ 
&= \sum_{T\in\bS T}\sum_{i=0}^2
    -\int_T(\partial_{\b t_{i}}
          \partial_{\b t_{i+1}} \partial_{\b t_{i-1}}
         \Big(\sum_{i=0}^2\partial_{\b t_i} u\Big)   e_h d\b x\\
 &=0,
   }  }
where $\sum \int_{\b x_i\b x_{i+1}}(\partial_{\b t_{i}} 
          \partial_{\b t_{i-1}}^2 u) e_h   d\b s =0$ as $\b t_i$ are parallel but in
  opposite directions on the two sides of the edge $\b x_i\b x_{i+1}$,
  $\sum \int_{\b x_{i+1}\b x_{i-1}}(\partial_{\b t_{i}} 
          \partial_{\b t_{i-1}}^2 u) e_h   d\b s =0$ as $\b t_i$ lies on the edge
   but in opposite directions on the two sides of the edge $\b x_i\b x_{i+1}$,
   and $\sum_{i=0}^2\partial_{\b t_i} u=0$ as $\sum_{i=0}^2 {\b t_i} =\b 0$.
By \eqref{m1}, \eqref{m-i-2} and \eqref{m2},  the theorem is proved.
\end{proof}

\begin{theorem} Let $u\in H^{4,\infty}(\Omega)  \cap H^1_0(\Omega)$
      and $u_h\in V_h$ of \eqref{V-h} 
  be the solutions of Poisson's equation \eqref{p-e} and
  the finite element equation \eqref{f-e}, respectively.
  It holds that
\a{      \| I_h u-u_h\|_{L^2}\le
     C\| I_h u-u_h\|_{L^\infty}\le C h^4|u|_{H^{4,\infty}}. } 
\end{theorem}

\begin{proof}  
  Let $g^j_h\in V_h$ be the discrete Green function at an interior vertex 
   $\b x_j\in \mathcal{N}_h$  
  satisfying
\a{ (\nabla g^j_h, \nabla v_h) = v_h(\b x_j) \quad \forall v_h\in  {V}_h.  }
Because of \eqref{o}, we have, similar to \eqref{m},
\a{ \ad{ (I_h u-u_h)(\b x_j) &= (\nabla(I_h u-u), \nabla g^j_h)  \\
    &= \sum_{T\in\mathcal{T}_h^6} (\nabla d_h, \nabla g^j_h)_T + 
       \sum_{T\in\bS T} (\nabla(\overline{I}_h u-u), \nabla g^j_h)_T  \\
    & = \sum_{T\in\mathcal{T}_h}\sum_{i=0}^2
       \int_{\b x_i\b x_{i+1}}  (I_h u-u) \partial_{\b t_{i+2}^\perp} g^j_h d\b s, }
} where $d_h$ is defined in \eqref{d-h},  and other notations are
  defined in \eqref{m}.
  
Repeating the proof from \eqref{m1} to \eqref{m2}, i.e., replacing $e_h$ there by
  $g^j_h$,  we get
\an{\label{i-1} |(I_h u- u_h)(\b x_j)| & = \bigg|
        \sum_{T\in\bS T} h^4 \int_{\partial T}(P_4)(\partial_{\b t}^4 u)(\partial_{\b t}
                 g^j_h) d\b s \bigg| \\
              &\le C h^4 |u|_{H^{4,\infty}}
   \bigg(\sum_{T\in\mathcal{T}_h} |g^j_h|_{H^1(T)}^2 \bigg)^{1/2}\\
              & \le C h^4 |u|_{H^{4,\infty}} |\overline{g}_h^j|_{H^1},}
where $\b t$ is a unit tangential vector on each edge of $T$,
  and $\overline{g}_h^j$ is the discrete function in $\overline{V}_h$, i.e., for a $\b x_j\in
   \bS N$,
\a{ (\nabla \overline{g}_h^j,\nabla  v_h) = v_h(\b x_j) \quad\forall v_h\in\overline{V}_h. }
Therefore, for a $\b x_j\in \mathcal{N}_h$, as $V_h\subset \overline{V}_h$, 
\a{ (\nabla {g}_h^j,\nabla v_h) = v_h(\b x_j)=(\nabla \overline{g}_h^j, \nabla v_h) 
             \quad\forall v_h\in {V}_h. }
It implies that ${g}_h^j$ is an $H^1$-orthogonal projection of $\overline{g}_h^j$
   in $V_h$.
By \eqref{i-1} and the $H^{1,\infty}$ stability of $\overline{g}_h^j$ from
  \cite{Rannacher-Scott}, we have 
\a{ |I_h u-u_h|_{L^\infty} &= \max_{\b x_j\in \mathcal{N}_h}
           |(I_h u-u_h)(\b x_j)| \\&\le 
  C h^4  |u|_{H^{4,\infty} } |{g}_h^j|_{H^1} \\&\le 
  C h^4  |u|_{H^{4,\infty} } |\overline{g}_h^j|_{H^1} \\&\le 
  C h^4  |u|_{H^{4,\infty} } |\overline{g}_h^j|_{H^{1,\infty}} \\&\le 
  C h^4  |u|_{H^{4,\infty} }. }
Consequently,
\a{ \|I_h u- u_h\|_{L^2}   \le 
    \t{meas}(\Omega) \|I_h u- u_h\|_{\infty}  \le 
  C h^4  |u|_{H^{4,\infty} }. } The proof is complete.
\end{proof}

                \vskip .7cm
   
\section{Lift the $P_1$ solution to a $P_3$ solution}
     
We can choose a larger patch, and/or several patterns of patches, for 
  locally lifting a solution.
For a simple way of computation, we choose only one patch $R$ of shape of
   one regular triangle of edge-size $2h$ on one honeycomb 
   mesh $\mathcal{T}_h$, cf. Figure \ref{f-6tr}.

\begin{figure}[ht] \setlength{\unitlength}{0.7pt}
 \begin{center}\begin{picture}(  430.,  200)( -100.,  0.)
     \def\lb{\circle*{0.8}}\def\lc{\vrule width1.2pt height1.2pt}
     \def\la{\circle*{0.3}}
 \put(-80,170){$R:$}  \put(106,170){$\mathcal{R}_h:$}
 \put(-5,6){$e_1$}  \put(50,100){$e_2$} \put(-60,100){$e_3$} 
     \multiput(-100.00,   0.00)(   0.250,   0.000){800}{\la}
     \multiput(-100.00,   0.00)(   0.125,   0.217){800}{\la}
     \multiput( 100.00,   0.00)(  -0.125,   0.217){800}{\la}
     \multiput(  23.00, 129.90)(  -3.000,   0.000){ 16}{\la}
     \multiput( -49.00,  84.87)(   1.500,  -2.598){ 16}{\la}
     \multiput( -49.00,   1.73)(   1.500,   2.598){ 16}{\la}
     \multiput(  49.00,  84.87)(  -1.500,  -2.598){ 16}{\la}
     \multiput(  49.00,   1.73)(  -1.500,   2.598){ 16}{\la}
     \multiput( -23.00,  43.30)(   3.000,   0.000){ 16}{\la}
   \put(100,-110){\includegraphics[width=190pt]{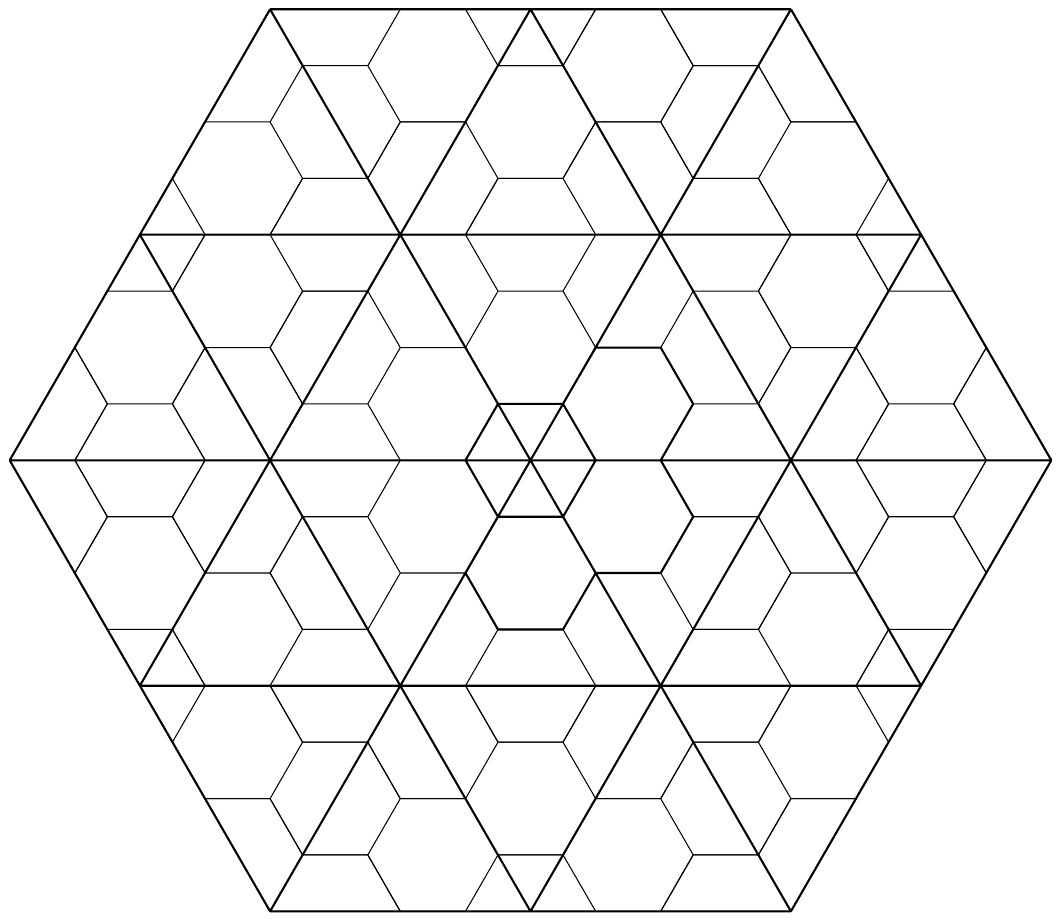}}

 \end{picture}\end{center}\caption{Left: A recovery patch $R$ of an equilateral triangle;
 \quad Right: a recovery patch grid $\mathcal{R}_h=\{ \t{regular triangle } \ R \}$
      on a regular hexagon domain. }   
\label{f-6tr} 
\end{figure}  

On each patch $R$, there are 11 vertices $\{\b x_j, \ j=1,\dots, 11\}\subset \mathcal{N}_h$
    of $\mathcal{T}_h$, cf. Figure \ref{f-6tr}.
Let $\tilde u_h\in P_3(R)$ be the least squares solution of following
   $11\times 10$ linear system of equations, 
\an{\label{lift} \tilde u(\b x_j) = u_h(\b x_j), \quad \ j=1,\dots, 11, }
where $u_h\in V_h$ is the $P_1$ solution of \eqref{f-e}.

\begin{lemma} The linear system \eqref{lift} of equations has a unique solution.
\end{lemma}

\begin{proof} The coefficient matrix of \eqref{lift} is of size $11\times 10$, as
      $\dim P_3(R)=10$.  To show it is a full-rank matrix,  we can show $\tilde u_h=0$
   if  \a{ \tilde u_h(\b x_j) = 0, \quad \ j=1,\dots, 11.}

 Because the $P_3$ polynomial $\tilde u_h$ vanishes at four points on edge $e_1$, 
  cf. Figure \ref{f-6tr}, it vanishes on the whole edge $e_1$.
 Thus \a{ \tilde u_h = l_1 p_2, }
where $p_2\in P_2(R)$ and $l_1\in P_1(R)$ which vanishes only on $e_1$.

Because $\tilde u_h$ vanishes but $l_1$ does not vanish at three top points
  of edge $e_2$, 
  cf. Figure \ref{f-6tr}, $p_2$ vanishes on the whole edge $e_2$.
 Thus \a{ \tilde u_h = l_1 l_2 p_1, }
where $p_1\in P_1(R)$ and $l_2\in P_1(R)$ which vanishes only on $e_2$.

Because $\tilde u_h$ vanishes but $l_1$ and $l_2$ do not vanish at two internal points
  of edge $e_3$, 
  cf. Figure \ref{f-6tr}, $p_1$ vanishes on the whole edge $e_3$.
 Thus \a{ \tilde u_h = C l_1 l_2 l_3, }
where $C\in P_0(R)$ and $l_3\in P_1(R)$ which vanishes only on $e_3$.

There are two internal $\b x_j$ inside $R$. 
Because $\tilde u_h(\b x_j)=0$, $l_1(\b x_j)\ne 0$, $l_2(\b x_j)\ne 0$ 
    and $l_3(\b x_j)\ne 0$, 
  cf. Figure \ref{f-6tr}, $C=0$, i.e. $\tilde u_h=0$.

Thus the   matrix in \eqref{lift} is of full rank and the
 solution $\tilde h$ is well defined uniquely by the least squares method.
\end{proof}

\begin{theorem} Let $u\in H^1_0(\Omega)\cap H^{4,\infty}(\Omega)$ and $\tilde u_h$ 
  be the solutions of Poisson's equation \eqref{p-e} and
  the lifting equation \eqref{lift}, respectively.
  It holds that
\a{  \| u- \tilde  u_h \|_{L^2}+ h| u- \tilde  u_h  |_{H^1,h} 
       \le C h^4|u|_{H^{4,\infty}}, }
where  $| u- \tilde  u_h  |_{H^1,h}$ is defined piecewise on $\mathcal{R}_h$.
\end{theorem}

\begin{proof} Let $\tilde I_h u=\tilde u_h\in P_3(R)$ be the least squares solution of
 \eqref{lift}.
Because the interpolation $\tilde I_h u$ preserves $P_3$ polynomials and is stable
  (cf. \cite{Scott-Zhang}), we have
\a{ \| u- \tilde  I_h u \|_{L^2}+ h| u- \tilde  I_h u  |_{H^1,h} 
       \le C h^4|u|_{H^{4}}. }
As $\tilde u_h\in  P_3(R)$, $\tilde I_h \tilde u_h = \tilde u_h$.
By definition \eqref{lift}, $\tilde I_h u=\tilde I_h I_h u$.
Thus, as $\tilde I_h$ is stable (cf. \cite{Scott-Zhang}), 
\a{ &\quad \ \| u- \tilde  u_h \|_{L^2}+ h| u- \tilde  u_h  |_{H^1,h} \\
    & \le \| u- \tilde I_h u \|_{L^2}+\|\tilde I_h(I_h u-  u_h) \|_{L^2} \\ &\qquad + 
       h| u- \tilde  u_h  |_{H^1,h} + h |\tilde I_h(I_h u-  u_h) \|_{H^1,h}\\
    & \le C h^4|u|_{H^{4,\infty}} + C\| I_h u-  u_h  \|_{L^2}+ 
       C h^4|u|_{H^{4,\infty}} + C h | I_h u-  u_h |_{H^1,h}\\
    & \le C h^4|u|_{H^{4,\infty}}. }
The theorem is proved.
\end{proof}

                \vskip .7cm

\section{Numerical test}

We solve the Poisson equation \eqref{p-e},
   on a regular hexagon domain $\Omega$ of size $1$ centered at the origin.
    The exact solution of \eqref{p-e} for the numerical test is
\an{\label{s-1} \ad{u(x,y)&=x^2\sin(\frac{\pi}2(\frac y{\sqrt 3}
     +x+1))\sin\frac{\pi}2(\frac y{\sqrt 3}
     -x+1))\\
   &\qquad \cdot \sin(\frac{\pi}{\sqrt 3}( y+\frac {\sqrt 3}2 )). }  }

The first three grids are depicted in Figure \ref{f-grid}.
Here the refinement can not be nested.
Some honeycombs are split and combined with some neighbor honeycombs' debris to
  form half-sized honeycombs.
The higher level grids are obtained by such half-sized refinements.
 
In Table \ref{t-1}, we list the error between the interpolation and the $P_1$
   virtual element
  solution to show two-order or three-order superconvergence.
The optimal order of convergence is $2$ in $L^2$ and $L^\infty$ norms, but we have an
  order $4$ for them in Table \ref{t-1}, i.e., order-two superconvergence.
The optimal order of convergence is $1$ for $P_1$ elements in $H^1$-norm, but we have an
  order $4$ for it in Table \ref{t-1}, i.e., order-three superconvergence. 
The results match the theory on the superconvergence.  
In Table \ref{t-1}, on first level, there is no internal degree of freedom in $V_h$ that
  the errors are supposedly zero.   They are round-off errors when evaluating
  the zero boundary condition.

\begin{table}[ht]
  \centering  \renewcommand{\arraystretch}{1.2}
  \caption{Error profile of $u_h$ on Figure \ref{f-grid} meshes  for  \eqref{s-1}. }
  \label{t-1}
\begin{tabular}{c|cc|cc|cc}
\hline
Grid &   $\|I_h u-u_h\|_{L^2}$  &  $O(h^r)$ 
   &   $|I_h u-u_h|_{H^1}$ & $O(h^r)$ &   $\|I_h u-u_h\|_{L^\infty}$  &  $O(h^r)$ \\
\hline 
 1&  0.1233E-31& 0.00&  0.2013E-31& 0.00&  0.1735E-31& 0.00 \\
 2&  0.1567E-02& 0.00&  0.4800E-02& 0.00&  0.1829E-02& 0.00 \\
 3&  0.1138E-03& 3.78&  0.3305E-03& 3.86&  0.1480E-03& 3.63 \\
 4&  0.7358E-05& 3.95&  0.2118E-04& 3.96&  0.9840E-05& 3.91 \\
 5&  0.4641E-06& 3.99&  0.1332E-05& 3.99&  0.6244E-06& 3.98 \\
 6&  0.2907E-07& 4.00&  0.8337E-07& 4.00&  0.3917E-07& 3.99 \\
 7&  0.1818E-08& 4.00&  0.5213E-08& 4.00&  0.2450E-08& 4.00 \\
 8&  0.1137E-09& 4.00&  0.3258E-09& 4.00&  0.1532E-09& 4.00 \\
 9&  0.7104E-11& 4.00&  0.2037E-10& 4.00&  0.9575E-11& 4.00 \\
10&  0.4440E-12& 4.00&  0.1279E-11& 3.99&  0.5985E-12& 4.00 \\
\hline
    \end{tabular}%
\end{table}%

In Table \ref{t-2}, we list the $L^2$-error between the solution and the
   $P_1$ virtual element
  solution, showing it can only converge at the optimal order in $L^2$-norm, order 2.
We also list the $L^2$ and the $H^1$ errors between the exact solution and the
  $P_3$ post-processed solution.  
The results match the theory on the optimal-order convergence of the lifted solution.  
In Table \ref{t-2}, we start to do lifting from level three.

\begin{table}[ht]
  \centering  \renewcommand{\arraystretch}{1.2}
  \caption{Error profile  of $u_h$ and $\tilde u_h$
        on Figure \ref{f-grid} meshes  for  \eqref{s-1}. }
  \label{t-2}
\begin{tabular}{c|cc|cc|cc}
\hline
Grid &   $\|u-u_h\|_{L^2}$  &  $O(h^r)$ &   $\|u-\tilde u_h\|_{L^2}$  &  $O(h^r)$ 
   &   $| u-\tilde u_h|_{H^1,h}$ & $O(h^r)$   \\
\hline 

 1&  0.8686E-01& 0.00&  0.0000E+00& 0.00&  0.0000E+00& 0.00 \\
 2&  0.3127E-01& 1.47&  0.0000E+00& 0.00&  0.0000E+00& 0.00 \\
 3&  0.9417E-02& 1.73&  0.3298E-02& 0.00&  0.2263E-01& 0.00 \\
 4&  0.2448E-02& 1.94&  0.2324E-03& 3.83&  0.2648E-02& 3.10 \\
 5&  0.6179E-03& 1.99&  0.1464E-04& 3.99&  0.3256E-03& 3.02 \\
 6&  0.1548E-03& 2.00&  0.9192E-06& 3.99&  0.4060E-04& 3.00 \\
 7&  0.3873E-04& 2.00&  0.5752E-07& 4.00&  0.5072E-05& 3.00 \\
 8&  0.9684E-05& 2.00&  0.3596E-08& 4.00&  0.6339E-06& 3.00 \\
 9&  0.2421E-05& 2.00&  0.2248E-09& 4.00&  0.7924E-07& 3.00 \\
10&  0.6053E-06& 2.00&  0.1405E-10& 4.00&  0.9905E-08& 3.00 \\
\hline
    \end{tabular}%
\end{table}%

\section{Ethical Statement}

\subsection{Compliance with Ethical Standards} { \ }

   The submitted work is original and is not published elsewhere in any form or language.

\subsection{Funding } { \ }

Yanping Lin is supported in part by HKSAR GRF 15302922  and polyu-CAS joint Lab.

Xuejun Xu is supported by National Natural Science Foundation of China (Grant
Nos. 12071350), Shanghai Municipal Science and Technology Major Project No.
2021SHZDZX0100, and Science and Technology Commission of Shanghai Municipality.

\subsection{Conflict of Interest} { \ }

  There is no potential conflict of interest.

\subsection{Ethical approval} { \ }

  This article does not contain any studies involving animals.
This article does not contain any studies involving human participants.
  
\subsection{Informed consent}  { \ }

This research does not have any human participant.  

\subsection{Availability of supporting data } { \ }

This research does not use any external or author-collected data.

\subsection{Authors' contributions } { \ }

All authors made equal contribution.
  
\subsection{Acknowledgments } { \ }

  None.

\end{document}